\documentclass[a4]{article}

\usepackage{enumerate}
\usepackage{amsmath,amssymb,amsthm}
\usepackage{ascmac}
\usepackage{eepic}
\usepackage{color}

\newtheorem{thm}{Theorem}[section]
\newtheorem{prop}[thm]{Proposition}

\numberwithin{equation}{section}

\DeclareFontEncoding{OT2}{}{}
\DeclareFontSubstitution{OT2}{wncyr}{m}{n}
\DeclareSymbolFont{cyss}{OT2}{wncyss}{m}{n}
\DeclareMathSymbol{\sh}{\mathbin}{cyss}{`x}

\newcommand{\longr}{\longrightarrow}

\newcommand{\C}{\mathbf{C}}

\newcommand{\bP}{{\mathbf P}}
\newcommand{\bH}{{\mathbf H}}
\newcommand{\bD}{{\mathbf D}}

\DeclareMathOperator{\Li}{\mathrm{Li}}

\title{The hexagon equations for dilogarithms and \\ the Riemann-Hilbert problem}
\author{Shu Oi\thanks{Department of Mathematics, College of Science, Rikkyo university. \endgraf \hspace{1em}3-34-1, 
Nishi-Ikebukuro, Toshima-ku, Tokyo 171-8501, Japan. \endgraf \hspace{1em}e-mail: {\tt shu-oi@rikkyo.ac.jp}}
\and Kimio Ueno\thanks{Department of Mathematics, School of Fundamental Sciences and Engineering,\endgraf \hspace{1em} Faculty of Science and Engineering, Waseda University. \endgraf \hspace{1em}3-4-1, Okubo, Shinjuku-ku, 
Tokyo 169-8555, Japan. \endgraf \hspace{1em}e-mail: {\tt uenoki@waseda.jp}
}}
\date{}

\begin{document}

\insert\footins{\footnotesize 2010 {\it Mathematics Subject Classification.} Primary 34M50,11G55; Secondary 30E25,11M06,32G34;}
\maketitle

\begin{abstract}
In this article we present the hexagon equations for dilogarithms which comes from the analytic continuation 
of the dilogarithm $\Li_2(z)$ to $\bP^1 \setminus \{0,1,\infty\}$. 
The hexagon equations are equivalent to the coboundary relations for a certain 1-cocycle of holomorphic functions 
on $\bP^1$, and are solved by the Riemann-Hilbert problem of additive type.
They uniquely characterize the dilogarithm under the normalization condition. 
\end{abstract}

\section{Introduction}\label{sec:intro}
Let $\bD_{\alpha} \ (\alpha=0,1,\infty)$ be domains in the Riemann sphere $\bP^1$ defined by
\begin{align*}
      & \bD_0  = \C \setminus \{z=x | 1 \leq x < +\infty \}, \
       \bD_1  = \C \setminus \{z=x | -\infty < x  \leq 0 \}, \\
      & \bD_{\infty} = \bP^1 \setminus \{z=x | 0 \leq x \leq 1\}, 
\end{align*}
which are open neighborhoods of the points $0,1,\infty$, respectively. They give 
an open covering of the Riemann sphere $\bP^1 = \bD_0 \cup \bD_1 \cup \bD_{\infty}$, furthermore, satisfy 
$\bD_0 \cap \bD_1 \cap \bD_{\infty} = \bH_+ \cup \bH_-$  where $\bH_+$ ($\bH_-$) is the upper (lower) half plane.

According to Theorem 2 in \cite{OiU1}, the dilogarithms $\Li_2(z)$ and $\Li_2(1-z)$ are characterized 
as the solutions $f_0(z)$ and $f_1(z)$ holomorphic in $\bD_0$ and $\bD_1$, respectively, to the following 
functional equation
\begin{align}\label{eq:dilog_RH}
      f_0(z) + \log z \log(1-z) + f_1(z) = \zeta(2) \quad (z \in \bD_0 \cap \bD_1) 
\end{align}
under the asymptotic condition
$f'_{\alpha}(z) \to 0 \ (z \to \infty, \ z \in \bD_{\alpha}, \alpha=0,1)$ and the normalization condition $f_0(0)=0$.

The equation \eqref{eq:dilog_RH} says that the function 
$-\log z \log(1-z) + \zeta(2)$, 
which is holomorphic in $\bD_0 \cap \bD_1$, decomposes to the sum of
 $f_0(z)$ and $f_1(z)$ holomorphic in $\bD_0$ and $\bD_1$. 
That is to say, \eqref{eq:dilog_RH} is a Riemann-Hilbert problem  (or, Plemelj-Birkhoff decomposition) of additive type
\cite{Bi,Mu,Pl}.
However the previous asymptotic condition does not naturally come from the Riemann-Hilbert problem, 
is a rather technical one. So we would like to avoid that, and give more natural formulation of 
the Riemann-Hilbert problem.

\section{The hexagon relations for dilogarithms}
Let $\log(1-z)$ be the principal value of the logarithm in $\bD_0$, namely, $\log 1 =0$. 
Define a branch of the dilogarithm $\Li_2(z)$ in $\bD_0$ by 
\begin{equation}
\Li_2(z) = - \int_0^z \frac{\log(1-t)}{t} \, dt.
\end{equation}
Then the Taylor expansion at $z=0$ is 
\begin{equation}
\Li_2(z) = \sum_{n=1}^{\infty} \frac{z^n}{n^2} \quad (|z|<1)
\end{equation}
and by Abel's continuous theorem, we have 
\begin{equation}
\lim_{z \to 1,z \in \bD_0}\Li_2(z)=\zeta(2). 
\end{equation}

It is well known that the dilogarithm satisfies the following functional relations \cite{Le}; 
\begin{align}
    & \Li_2(z) + \log z \log(1-z) + \Li_2(1-z) = \zeta(2) \quad (z \in \bD_0 \cap \bD_1),
                                                                      \label{eq:dilog_euler1} \\
    & \Li_2\big( \frac{z}{z-1} \big) = -\Li_2(z) - \frac{1}{2} \log^2(1-z) \quad (z \in \bD_0). 
                                                                             \label{eq:dilog_landen}
\end{align}
The first formula describes the direct analytic continuation of $\Li_2(z)$ along the path $\overline{01}$, 
and the second one is regarded as the half-monodromy continuation of $\Li_2(z)$ at $z=0$. 
We call them Euler's inversion formula and Landen's inversion formula, respectively. 

To consider the global analytic continuation of the dilogarithm, let us introduce 
automorphisms $t_i(z)$ of $\bP^1$ as follows:
\begin{align}\label{eq:P1auto}
      \left\{
      \begin{array}{cll}
         t_0(z)=z_0=z, & \quad  t_1(z)=z_1=1-z,  & \quad  t_2(z)=z_2=\frac{z-1}{z},  \\
         t_3(z)=z_3=\frac{1}{z}, & \quad  t_4(z)=z_4=\frac{1}{1-z}, & \quad  t_5(z)=z_5=\frac{z}{z-1}. 
      \end{array}
      \right.
\end{align}
They induce permutations $\sigma_{t_i}$ of the points $\{0, 1, \infty\}$, 
\begin{align*}
      & \sigma_{t_0}=\begin{pmatrix}0&1&\infty\\0&1&\infty\end{pmatrix}, & \
      &  \sigma_{t_1} =\begin{pmatrix}0&1&\infty\\1&0&\infty\end{pmatrix},& \
      &  \sigma_{t_2}  =\begin{pmatrix}0&1&\infty\\ \infty&0&1\end{pmatrix}, & \\
      & \sigma_{t_3}=\begin{pmatrix}0&1&\infty\\ \infty &1&0\end{pmatrix}, & \
      &  \sigma_{t_4}  =\begin{pmatrix}0&1&\infty\\1&\infty&0\end{pmatrix}, & \
      & \sigma_{t_5}  =\begin{pmatrix}0&1&\infty\\0&\infty&1\end{pmatrix}, & 
\end{align*}
and give analytic isomorphisms \ $t_i : \bD_{\alpha} \ \longr \ \bD_{t_i(\alpha)}$. Furthermore, we should note that
\begin{align*}
        z_{j-1}=\frac{z_j}{z_j-1}, \quad  z_{j+1}=1-z_j, \  \quad (j=0,2,4) 
\end{align*}
where the indices are defined modulo 6.
Taking the inverse image of the relations 
\eqref{eq:dilog_euler1} and \eqref{eq:dilog_landen} induced by $t_j \ (j=0,2,4)$, 
we have the following proposition:

\begin{prop}\label{prop:hexagon}
The dilogarithm $\Li_2(z)$ satisfies
\begin{align}
  & \Li_2(z_j) + \log z_j \log z_{j+1} + \Li_2(z_{j+1}) = \zeta(2) \quad 
  (z \in \bD_{t_j^{-1}(0)} \cap \bD_{t_j^{-1}(1)}), \label{eq:hexagon0} \\
  & \Li_2(z_{j-1}) = - \Li_2(z_{j}) - \frac{1}{2} \log^2z_{j+1}  \quad (z \in \bD_{t_{j}^{-1}(0)}), 
                                                                             \label{eq:hexagon1} 
\end{align}
where $j=0,2,4$, and all the indices are defined modulo 6.
\end{prop}

We refer to them as the hexagon relations for dilogarithms, which describe the analytic continuation of $\Li_2(z)$ 
along the paths through the neighborhood of $0 \to 1 \to \infty \to 0$ in $\bH_{\pm}$ (see the figures below).

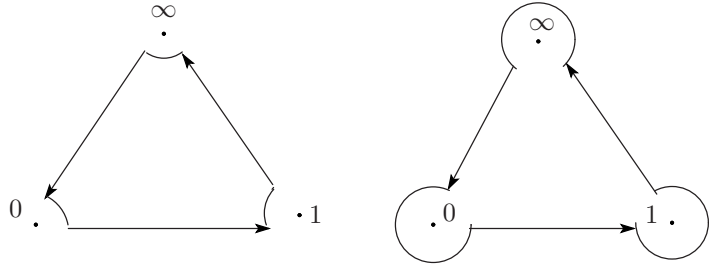
\begin{center}
\unitlength 0.1in
\begin{picture}( 38.1000, 13.9500)( 11.8000,-28.1000)
%
{\color[named]{Black}{%
\special{pn 4}%
\special{sh 1}%
\special{ar 2810 2550 8 8 0  6.28318530717959E+0000}%
\special{sh 1}%
\special{ar 2810 2550 8 8 0  6.28318530717959E+0000}%
}}%
%
{\color[named]{Black}{%
\special{pn 4}%
\special{sh 1}%
\special{ar 1430 2600 8 8 0  6.28318530717959E+0000}%
\special{sh 1}%
\special{ar 1430 2600 8 8 0  6.28318530717959E+0000}%
}}%
%
{\color[named]{Black}{%
\special{pn 8}%
\special{pa 2000 1690}%
\special{pa 1480 2460}%
\special{fp}%
\special{sh 1}%
\special{pa 1480 2460}%
\special{pa 1534 2416}%
\special{pa 1510 2416}%
\special{pa 1502 2394}%
\special{pa 1480 2460}%
\special{fp}%
}}%
\put(13.6000,-24.7000){\makebox(0,0)[rt]{$0$}}%
\put(28.6000,-25.9000){\makebox(0,0)[lb]{$1$}}%
\put(21.0000,-14.8000){\makebox(0,0){$\infty$}}%
%
{\color[named]{Black}{%
\special{pn 4}%
\special{sh 1}%
\special{ar 4760 2590 8 8 0  6.28318530717959E+0000}%
\special{sh 1}%
\special{ar 4760 2590 8 8 0  6.28318530717959E+0000}%
}}%
%
{\color[named]{Black}{%
\special{pn 4}%
\special{sh 1}%
\special{ar 4060 1640 8 8 0  6.28318530717959E+0000}%
\special{sh 1}%
\special{ar 4060 1640 8 8 0  6.28318530717959E+0000}%
}}%
%
{\color[named]{Black}{%
\special{pn 8}%
\special{ar 4060 1630 192 192  2.2142974 6.2831853}%
\special{ar 4060 1630 192 192  0.0000000 0.8224183}%
}}%
%
{\color[named]{Black}{%
\special{pn 8}%
\special{pa 3700 2620}%
\special{pa 4560 2620}%
\special{fp}%
\special{sh 1}%
\special{pa 4560 2620}%
\special{pa 4494 2600}%
\special{pa 4508 2620}%
\special{pa 4494 2640}%
\special{pa 4560 2620}%
\special{fp}%
}}%
%
{\color[named]{Black}{%
\special{pn 8}%
\special{pa 4680 2440}%
\special{pa 4210 1760}%
\special{fp}%
\special{sh 1}%
\special{pa 4210 1760}%
\special{pa 4232 1826}%
\special{pa 4240 1804}%
\special{pa 4264 1804}%
\special{pa 4210 1760}%
\special{fp}%
}}%
%
{\color[named]{Black}{%
\special{pn 8}%
\special{pa 3930 1770}%
\special{pa 3590 2410}%
\special{fp}%
\special{sh 1}%
\special{pa 3590 2410}%
\special{pa 3640 2362}%
\special{pa 3616 2364}%
\special{pa 3604 2342}%
\special{pa 3590 2410}%
\special{fp}%
}}%
%
{\color[named]{Black}{%
\special{pn 4}%
\special{sh 1}%
\special{ar 3510 2600 8 8 0  6.28318530717959E+0000}%
\special{sh 1}%
\special{ar 3510 2600 8 8 0  6.28318530717959E+0000}%
}}%
%
{\color[named]{Black}{%
\special{pn 8}%
\special{ar 3510 2600 198 198  6.2831853 6.2831853}%
\special{ar 3510 2600 198 198  0.0000000 5.1522316}%
}}%
\put(35.6000,-25.8000){\makebox(0,0)[lb]{$0$}}%
\put(46.9000,-25.9000){\makebox(0,0)[rb]{$1$}}%
\put(40.8000,-15.5000){\makebox(0,0){$\infty$}}%
%
{\color[named]{Black}{%
\special{pn 8}%
\special{pa 1600 2620}%
\special{pa 2650 2620}%
\special{fp}%
\special{sh 1}%
\special{pa 2650 2620}%
\special{pa 2584 2600}%
\special{pa 2598 2620}%
\special{pa 2584 2640}%
\special{pa 2650 2620}%
\special{fp}%
}}%
%
{\color[named]{Black}{%
\special{pn 8}%
\special{pa 2670 2390}%
\special{pa 2200 1710}%
\special{fp}%
\special{sh 1}%
\special{pa 2200 1710}%
\special{pa 2222 1776}%
\special{pa 2230 1754}%
\special{pa 2254 1754}%
\special{pa 2200 1710}%
\special{fp}%
}}%
%
{\color[named]{Black}{%
\special{pn 8}%
\special{ar 1390 2640 212 212  5.1760366 6.0978374}%
}}%
%
{\color[named]{Black}{%
\special{pn 4}%
\special{sh 1}%
\special{ar 2100 1600 8 8 0  6.28318530717959E+0000}%
\special{sh 1}%
\special{ar 2100 1600 8 8 0  6.28318530717959E+0000}%
}}%
%
{\color[named]{Black}{%
\special{pn 8}%
\special{ar 2100 1570 158 158  0.9151007 2.1763410}%
}}%
%
{\color[named]{Black}{%
\special{pn 8}%
\special{ar 2830 2540 206 206  3.7944393 3.8556834}%
}}%
%
{\color[named]{Black}{%
\special{pn 8}%
\special{ar 2810 2540 184 184  2.8023000 3.8899707}%
}}%
%
{\color[named]{Black}{%
\special{pn 8}%
\special{ar 4760 2590 190 190  4.2130423 6.2831853}%
\special{ar 4760 2590 190 190  0.0000000 3.1415927}%
}}%
\end{picture}%

\end{center}

\section{The hexagon equations for dilogarithms and a 1-cocycle of holomorphic functions on $\bP^1$}
In the sequel, the index $j$ denotes $0,2,4$, and all the indices are defined modulo 6. 

Let $f_j(z),\ f_{j-1}(z)$ be functions holomorphic in  $\bD_{t_j^{-1}(0)}$, 
and consider the following functional equations for them:
\begin{align}
   & f_j(z) + \log z_j \log z_{j+1} + f_{j+1}(z) = \zeta(2) 
                              \quad \text{{\rm in}  $\bD_{t_j^{-1}(0)} \cap \bD_{t_j^{-1}(1)}$}, 
                                            \label{eq:dilog_RH0} \\
   & f_{j-1}(z) = - f_j(z) - \frac{1}{2} \log^2z_{j+1} \quad \text{{\rm in} $\bD_{t_j^{-1}(0)}$}.  \label{eq:dilog_RH1} 
\end{align}
We call them the hexagon equations for dilogarithms. 
We show that they are relevant to the cohomology theory on $\bP^1$:

\begin{prop}\label{prop:1_cocycle}
To the covering $\{\bD_0, \bD_1, \bD_{\infty}\}$ of $\bP^1$, we attach a 0-cochain 
$f=\{f_{\bD_{\alpha}}\}_{\alpha=0,1,\infty}$ and a 1-cochain 
$F=\{F_{\bD_{\alpha}\bD_{\beta}}\}_{\alpha, \beta = 0, 1, \infty}$ where $f_{\bD_{\alpha}}$ is a holomorphic 
function on $\bD_{\alpha}$ and $F_{\bD_{\alpha}\bD_{\beta}}$ is a holomorphic function on 
$\bD_{\alpha} \cap \bD_{\beta}$ defined by 
\begin{align}
       f_{\bD_{t_j^{-1}(0)}}(z) = f_j(z), \quad 
       F_{\bD_{t_j^{-1}(0)},\bD_{t_j^{-1}(1)}}(z) = \dfrac{1}{2}\log^2z _j- \log z_j \log z_{j+1} + \zeta(2), 
        \label{eq:1_cocycle} 
\end{align}
and $F_{\bD_{\beta}\bD_{\alpha}} = - F_{\bD_{\alpha}\bD_{\beta}}$. Then we have the following: 
\begin{enumerate}
\item $F$ is a 1-cocycle. That is to say, $F$ satisfies the cocycle condition
      \begin{align}\label{eq:cocycle_cond}
         F_{\bD_1\bD_{\infty}}(z) - F_{\bD_0\bD_{\infty}}(z) + F_{\bD_0\bD_1}(z)=0 
                           \quad (z \in \bD_0 \cap \bD_1 \cap \bD_{\infty}).
      \end{align}
\item The condition \eqref{eq:cocycle_cond} is equivalent to the Euler formula $\zeta(2)=\frac{\pi^2}{6}$. 
\item The hexagon equations \eqref{eq:dilog_RH0} and \eqref{eq:dilog_RH1} 
      are equivalent to \eqref{eq:dilog_RH1} and the coboundary relations 
\begin{align}
       F_{\bD_{t_j^{-1}(0)},\bD_{t_j^{-1}(1)}}(z)=f_{\bD_{t_j^{-1}(0)}}(z) - f_{\bD_{t_j^{-1}(1)}}(z) \quad 
                     (z \in \bD_{t_j^{-1}(0)} \cap \bD_{t_j^{-1}(1)}) \label{eq:cob1}.
\end{align}
\end{enumerate}
\end{prop}
\begin{proof}
From \eqref{eq:1_cocycle}, it follows that 
\begin{equation*}
            F_{\bD_1\bD_{\infty}}(z) - F_{\bD_0\bD_{\infty}}(z) + F_{\bD_0\bD_1}(z) 
            =3\zeta(2) +  
              \frac{1}{2} \left( \log z + \log\big(\frac{z-1}{z}\big) + \log\big(\frac{1}{1-z}\big) \right)^2. 
\end{equation*}
As the logarithms satisfy the half-monodromy relations 
\begin{align*}
         \log z + \log\big(\frac{z-1}{z}\big) + \log\big(\frac{1}{1-z}\big) = \pm\pi i \quad (z \in \bH_{\pm}), 
\end{align*}
we have $F_{\bD_1\bD_{\infty}}(z) - F_{\bD_0\bD_{\infty}}(z) + F_{\bD_0\bD_1}(z)
          = 3 \zeta(2) - \frac{\pi^2}{2}$. Thus the claims (i) and (ii) are proved. 

The claim (iii) immediately follows from $t^{-1}_{j+1}(0)=t^{-1}_j(1)$.
\end{proof}

\section{Main theorem}

From Proposition \ref{prop:1_cocycle} and the results on the cohomology theory that
$H^1(\bP^1, \mathcal{O}) = 0,\ H^0(\bP^1, \mathcal{O}) = \C$,  
where $\mathcal{O}$ denotes the sheaf of holomorphic functions on $\bP^1$, one can deduce that 
the solutions to the hexagon 
equations exist uniquely under the normalization condition $f_0(0)=1$.
However one cannot know how the solutions are relevant to dilogarithms. So let us solve the hexagon equation 
by the Riemann-Hilbert problem of additive type. 

\begin{thm}\label{main_thm}
The hexagon equations for dilogarithms have a unique solution under the normalization condition $f_0(0)=0$, 
and the solutions are expressed as 
\begin{align}
           f_i(z)=\Li_2(z_i) \quad (i=0,1,\dots,5). 
\end{align}
\end{thm}
\begin{proof}
Differentiating the equations 
\eqref{eq:dilog_RH0} and \eqref{eq:dilog_RH1} with respect to the 
variable $z=z_0$, we have the following:
\begin{align}
      & f_j'(z) + \log z_{j+1} (\log z_j)' = -f_{j+1}'(z) - \log z_j (\log z_{j+1})',  \label{eq:diff_dilog_RH0} \\
     & f_{j+1}'(z) + \log z_j (\log z_{j+1})' = -f_{j+2}'(z) - \log z_{j+3} (\log z_{j+2})'  \label{eq:diff_dilog_RH1}
\end{align}
where $j=0,2,4$. Put 
\begin{align*}
       f(z)=f_0'(z) + \log z_1 (\log z_0)'=f_0'(z)+\frac{\log(1-z)}{z}. 
\end{align*}
This is a function holomorphic in $\bD_0$. From \eqref{eq:diff_dilog_RH0}, it is 
analytically continued to $g(z)$ and $h(z)$ 
\begin{align*}
       g(z) &=- f_1'(z) -  \log z_0 (\log z_1)' = -f_1'(z) + \frac{\log z}{1-z},\\
       h(z) &= -f_3'(z) - \log z_2 (\log z_3)' = -f_3'(z) + \dfrac{\log\big(\frac{z-1}{z}\big)}{z}, 
\end{align*}
which are holomorphic in $\bD_1$ and $\bD_{\infty}$, respectively. 
Hence $f(z)$ is holomorphic on $\bP^1$ so that $f(z)=g(z)=h(z)=A$ where $A$ is a constant. Substitute this into 
$\eqref{eq:diff_dilog_RH0}$ and $\eqref{eq:diff_dilog_RH1}$, and integrate them by using
\begin{equation*}
    \log z_{j+1} (\log z_j)' = -\big(\Li_2(z_j)\big)', \qquad \log z_j (\log z_{j+1})' = -\big(\Li_2(z_{j+1})\big)'. 
\end{equation*}
Then we obtain 
\begin{align*}
        f_i(z)=\Li_2(z_i)+(-1)^iAz+b_i \quad (i=0,1,\dots,5).
\end{align*}
Here $b_i$ are integral constants. 
From the normalization condition $f_0(0)=0$, it is clear that $b_0=0$. 
Hence from \eqref{eq:dilog_RH0}, we have 
\begin{align*}
    \Li_2(z) + \log z \log(1-z) + \Li_2(1-z) + b_1 = \zeta(2).
\end{align*}
Taking the limit as $z \to 0 \ (z \in \bD_0 \cap \bD_1)$ in this equation, we have $b_1=0$ and $f_1(z)=\Li_2(1-z)-Az$. 
In a similar way, we show that all the constants $b_i$ are 0. Furthermore, since $f_3(z)$, which is expressed as
\begin{align*}
              f_3(z)=\Li_2\big(\frac{1}{z}\big)+Az, 
\end{align*}
should be holomorphic at $z=\infty$, we have $A=0$. Thus the proof is completed.
\end{proof}

\section{Discussions}

In \cite{OiU2}, the Riemann-Hilbert approach generalizes to the case of the fundamental solution normalized at 
the origin of the KZ equation of one variable. So it will be worth trying to generalize 
the hexagon equations approach to the case of the fundamental solutions 
of the KZ equation of one variable (cf. \cite{OkU}), and further to the case of the KZ equation of two variables 
(cf. \cite{OiU3}).

\end{document}